\documentclass[12pt]{amsart}
\usepackage[margin=1.2in]{geometry}
\usepackage{graphicx} 
\usepackage{dsfont} 
\usepackage{amsmath} 
\usepackage[backend=bibtex,style=alphabetic]{biblatex} 
\usepackage{amsmath,amsthm,amssymb,amsfonts,  fancyhdr, color, comment, graphicx, environ, mathtools}
\renewbibmacro{in:}{}
\DeclareFieldFormat[article]{volume}{\mkbibbold{#1}} 
\DeclareFieldFormat[article]{number}{\mkbibbold{#1}} 

\usepackage{hyperref} 
\hypersetup{        
    colorlinks=true,
    linkcolor=blue,
    filecolor=blue,      
    urlcolor=blue,
    pdfpagemode=FullScreen,
    }

\usepackage{tikz}
\usepackage{quiver}
\usepackage{adjustbox}
\usepackage{colortbl}
\usepackage{booktabs}

\usepackage{amsmath,amsthm,amssymb,amsfonts,fancyhdr,color, comment, graphicx, environ, mathtools}
\usepackage{xcolor}
\usepackage{mdframed}
\usepackage[shortlabels]{enumitem}
\usepackage{hyperref}
\usepackage{pgfplots}
\usepackage{blindtext}
\usepackage{setspace}

\hypersetup{
	colorlinks=true,
	linkcolor=blue,
	filecolor=magenta,      
	urlcolor=blue,
}
\usepackage{multicol,wrapfig,bbm}
\graphicspath{ {./images/} }

\pgfplotsset{compat=1.5}

\DeclareMathOperator{\preper}{Preper}

\newtheorem{mthm}{Theorem}

\theoremstyle{definition}
\newtheorem{theorem}{Theorem}[section]
\newtheorem{corollary}[theorem]{Corollary}
\newtheorem{lemma}[theorem]{Lemma}

\newtheorem{proposition}[theorem]{Proposition}

\newtheorem*{question}{Question}

\newtheorem*{remark}{Remark}

\usepackage{diagbox}

\begin{document}

\title{H\'enon maps with many rational periodic points}
\author[Kim]{Hyeonggeun Kim}
\email{hk523@cam.ac.uk}
\author[Krieger]{Holly Krieger}
\email{hkrieger@dpmms.cam.ac.uk}
\author[Postolache]{Mara-Ioana Postolache}
\email{mip32@cam.ac.uk}
\author[Szeto]{Vivian Szeto}
\email{ws397@cam.ac.uk}
\date{\today}

\begin{abstract}
    In this paper we study the dynamical uniform boundedness conjecture in the context of H\'enon maps.  Building on work of Doyle and Hyde on polynomial maps in one variable, we construct for each odd integer $d \geq 3$ a H\'enon map of degree $d$ defined over $\mathbb{Q}$ with at least $(d-4)^2$ integral periodic points and an integral cycle of length at least $\frac{8d+10}{3}$.   This provides a quadratic lower bound on any conjectural uniform bound for periodic rational points of H\'enon maps, and answers in the affirmative a question of Ingram concerning the existence of many rational periodic points and long rational cycles for H\'enon maps.

    \end{abstract}

\maketitle

\section{Introduction}

Fix an integer $d \geq 2$ and let $f: \mathbb{P}^n \rightarrow \mathbb{P}^n$ be a morphism of degree $d$ defined over a number field $K$.  We say a point $P$ is \emph{preperiodic} for $f$ if the forward orbit 
$$\{ P, f(P), f^{\circ 2}(P), \dots \}$$
of $P$ is a finite set. By a theorem of Northcott \cite{northcott}, the set $\preper(f,K)$ of $K$-rational preperiodic points of $f$ is finite. The \emph{dynamical uniform boundedness conjecture} proposed by Morton and Silverman \cite{Morton1994RationalPP} asserts that the size of $\preper(f,K)$ is bounded by a quantity depending only the dimension $n$, the degree $d$, and the extension degree $[K:\mathbb{Q}]$. 

Very little is known about this conjecture. By work of Fakhruddin \cite{fakhruddin}, the Morton-Silverman conjecture implies the analogous uniform boundedness conjecture on rational torsion points of abelian varieties, currently only known in the case of dimension $1$ \cite{mazur_1, merel}.  However, computational and conditional evidence suggests that it is indeed difficult to produce maps with large sets of rational preperiodic points (see \cite{Benedetto} and \cite{looper_dynamicallang} for some of the strongest known evidence).

In the case when $n = 1$ and $K=\mathbb{Q}$, a basic interpolation argument allows the construction of rational maps $f: \mathbb{P}^1 \rightarrow \mathbb{P}^1$ over $\mathbb{Q}$ with at least $d+1$ rational points.  Recently, Doyle and Hyde \cite{doyle2022polynomialsrationalpreperiodicpoints} provided in each degree $d \geq 2$ an explicit example $r_d(x)$ of a polynomial of degree at most $d$ which has at least $d+6$ rational preperiodic points, and gave a non-constructive proof that there exist examples in degree at most $d$ with $d + \lfloor \log_2(d) \rfloor$ rational preperiodic points, providing the best known lower bound on any uniform constant arising in the Morton-Silverman conjecture.  Surprisingly, DeMarco and Mavraki \cite{DeMarco_Mavraki_2024} have shown that the examples of Doyle-Hyde have applications to uniform unlikely intersections in complex dynamics.

H\'enon maps appear as a natural generalization of quadratic polynomials and are a fruitful entry point to studying polynomial dynamics in dimension greater than one, as initiated in \cite{hubbard}. Given an integer $d \geq 2,$ \emph{H\'enon map of degree $d$} is a map of the form 
$$(x,y) \mapsto (y, -\delta x + p(y))$$
where $p$ is a degree $d$ polynomial of one variable and $\delta$ a non-zero constant: note that a H\'enon map is a birational map of the plane with constant Jacobian determinant $\delta$. A H\'enon map does not extend to a morphism of the projective plane, as the natural extension has $[1,0,0]$ as a point of indeterminacy; however, any H\'enon map  nonetheless admits a dynamical height function and so the periodic points satisfy a Northcott finiteness property \cite{silverman_henon, kawaguchi_canonical_height}.  That is, for any degree $d$ H\'enon map $f$ defined over a number field $K$, the set $\mathrm{Per}(f,K)$  of periodic points with coordinates in $K$ is a finite set.  

In analogy with the conjecture of Morton and Silverman, one might hope that the size of the set of $K$-rational periodic points depends only on the degree of the H\'enon map and the extension degree $[K:\mathbb{Q}]$; Ingram \cite{ingram_henon_smallheight} has provided evidence for this in the case that $d = 2$ and $K = \mathbb{Q}$ by computing the set of rational periodic points for H\'enon maps of the form $(y, x +y^2+b)$ where $b \in \mathbb{Q}$ has small height.

In this article, we study the polynomials constructed by Doyle and Hyde and use them to construct an explicit family of H\'enon maps $\{ h_d(x,y) \}$ over $\mathbb{Q}$ with degree $d \rightarrow \infty$ and a large number of rational periodic points and long rational cycles.  This answers in the affirmative the first two parts of the following question of Ingram \cite{henon_openqs}:

\begin{question} [Ingram] Over a number field $K$, is it possible to construct infinite families of generalized H\'enon maps of algebraic degree $d$, and $K$-rational cycles of length at least $d+3$?  Can one construct maps with $N_d$ periodic points, where $N_d - d \rightarrow \infty$, or even $N_d/d \rightarrow \infty$ as $d \rightarrow \infty$?
\end{question}

\begin{mthm}\label{thm:manyperiodicpoints} For each odd degree $d >2$ there exists an explicitly constructed polynomial $s_d$ of degree at most $d$ with rational coefficients such that the H\'enon map 
$$h_d(x,y) = (y, -x + s_d(y))$$
has at least $(d-4)^2$ rational periodic points.
\end{mthm}

If $d = 2k+1,$ the constructed polynomials $s_d$ are integer-valued polynomials which satisfy 
$$\lim_{k \rightarrow \infty} (-1)^k s_{2k+1}(x) = \frac{2}{\sqrt{3}}\sin\left(\frac{\pi x}{3}\right) =: s_{\infty}(x,y).$$
A direct computation shows (see section \ref{subsection:transc_henon}) that any point of $\mathbb{C}^2$ with integer coordinates is periodic under iteration of the limiting H\'enon map $h_{\infty}(x,y)= (y,-x+s_{\infty}(y)$, with period $n \in \{ 1,4,5,6,12,20\}$ determined by the class of $(x,y)$ modulo $6$.  

Though the proof of Theorem \ref{thm:manyperiodicpoints} is not asymptotic, each map $h_d$ preserves the lattice $\mathbb{Z}^2$ and so one expects a large set of periodic integral points within the plane region where $s_d$ is a good approximation of $s_{\infty}$.  Interestingly, the dynamics near the boundary of this region allows for integer cycles of $h_d$ with long period, despite the bound of $20$ on the period of an integer cycle for the limiting map $h_{\infty}$: see Figure \ref{fig:longcycle}.

\begin{figure}[h!]
  \includegraphics[scale=0.4]{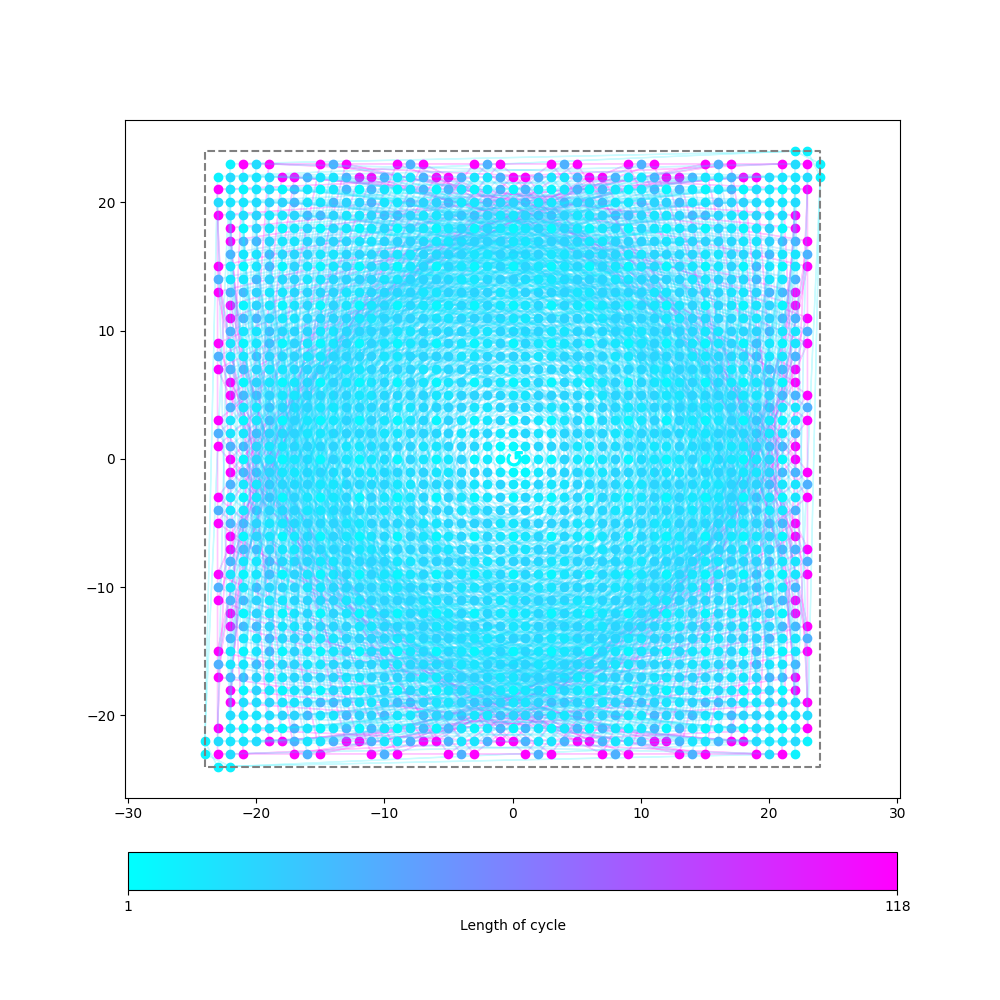}
  \caption{Periodic integer cycles for $h_d$ with $d = 43$ plotted in $\mathbb{Z}^2$ and colored by period length.  There are a large number of periodic points with cycle length at most $20$, but the boundary of this region admits a cycle of length $118$.}
  \label{fig:longcycle}
\end{figure}

\begin{mthm}\label{thm:longcycles} Let $d \equiv 1 \mod 6$ and $h_d(x,y) = (y, -x + s_d(y))$ as above.  Then $h_d$ has a cycle of integer points of length $\frac{8d+10}{3}.$
\end{mthm}

One can produce similar results in other odd degrees by introducing a shift to the polynomial $s_d$.  This is in striking contrast with the one-dimensional results of Doyle and Hyde, which are many-to-one on their sets of preperiodic integer points and hence produce only short cycles.  The family of degree $d$ H\'enon maps has dimension $d+2$ and so this theorem also improves substantially on what is obtained by interpolation to produce an integer cycle (this is still true in the larger family of generalized H\'enon maps, which can have mixed dimension (see \cite{FriMil}) if $d$ is a composite integer, but has maximal component dimension $d+6$). 

The organization of the paper is as follows.  In Section 2 we introduce the polynomials $s_d$ used to construct the H\'enon maps $h_d$, and in Section 3 we prove various auxiliary estimates which will allow us to understand $h_d$, as well as noting as an application an optimality result for Doyle-Hyde.  Section 4 provides the proof of Theorem \ref{thm:manyperiodicpoints}, and Theorem \ref{thm:longcycles} is proved in Section 5, which also includes an analysis of the limiting map $h_{\infty}.$

\subsection*{Acknowledgements} The work for this paper was done during a 2023 Summer Research In Mathematics (SRIM) programme in Cambridge.  The first author was supported by SRIM and the third and fourth authors by the Philippa Fawcett Internship Programme. The fourth author was also supported by the Trinity College Summer Studentship Scheme. We are grateful to the Faculty of Mathematics and the SRIM organizer Dhruv Ranganathan for their support. We also thank Romain Dujardin, Trevor Hyde, and Patrick Ingram for helpful conversations related to this work.

\section{Background}

In this section, we introduce the polynomials $s_d$ as integer-valued approximations of discrete trigonometric functions.  We also review the notion of dynamical compression and the constructions given in \cite{doyle2022polynomialsrationalpreperiodicpoints}.

\subsection{Compressing polynomials} Let $K$ be a number field, $f\in K[x]$ be a polynomial and $f^n := f^{\circ n}$ be its $n$-fold composition. Recall that $\alpha\in \overline{K}$ is \textit{preperiodic} if the forward orbit of $\alpha$ under $f$ is finite. Let $\preper(f,K)$ denote the set of preperiodic points in $K$. 

Let $[m] := \left\{1,2,\cdots, m\right\}$ for positive integer $m$. A polynomial $g\in \mathbb{C}[x]$ is said to exhibit \textit{dynamical compression} if it is linearly conjugate to some polynomial $f\in \mathbb{C}[x]$ which satisfies $f([m])\subseteq [m]$ for some $m>d+1$. 

\subsection{Discrete approximations of sine and cosine} Define the following functions of one variable:
$$ c_{2j}(x) := \frac{1}{(2j)!}\prod_{i=1}^{j} \left(x^2-\left(\frac{2i-1}{2}\right)^2\right),\ 
c_{2j+1}(x) := \frac{1}{(2j+1)!}x\prod_{i=1}^{k}(x^2-i^2).$$
We then define for any $k \geq 0$ the functions
$$ s_{2k}(x) := \sum_{j=0}^{k}(-1)^{k-j}c_{2j}(x)$$
and
$$
s_{2k+1}(x) := \sum_{j=0}^{k}(-1)^{k-j}c_{2j+1}(x).
$$

The degree $d$ polynomial $s_d\in \mathbb{Q}[x]$ is integer-valued on $\mathbb{Z}+\frac{d+1}{2}$ and takes values in $\{-1, 0, 1\}$ for integers in the interval $[-\frac{d+1}{2}, \frac{d+1}{2}].$ In \cite{doyle2022polynomialsrationalpreperiodicpoints}, Doyle and Hyde show that by applying the following transformation:
$$ r_d(x) = \begin{cases}
    s_d(x-3-\frac{d+1}{2})+2 & \text{if } d \text{ even} \\
    s_d(x-3-\frac{d+1}{2}) - x + d + 6 & \text{if } d \text{ odd}
\end{cases}
$$
one obtains a family of integer-valued polynomials $r_d\in\mathbb{Q}[x]$ that exhibit dynamical compression, namely:
$$ \begin{cases}
    r_d([d+6])\subseteq [d+5] & \text{if } d \text{ even} \\ 
    r_d([d+6])\subseteq [d+4] & \text{if } d \text{ odd}. \\ 
\end{cases}$$

The polynomial family $s_d$ is alternatively characterized by initial values together with the inductive relation
$$ \delta s_{d+1}(x) = s_d(x),\ s_1(x) = x$$
where $\delta$ is the \textit{central difference operator}:
$$ \delta f(x) = f\left(x+\frac{1}{2}\right) - f\left(x-\frac{1}{2}\right).$$
For $d \equiv 1 \mod 4,$ $s_d$ can be interpreted as a discrete analogue of the sine function on some finite interval; indeed, the construction of $s_d$ boils down to being a polynomial interpolation on the set $\left\{-\frac{d+1}{2},-\frac{d-1}{2},\cdots,\frac{d+1}{2}\right\}$ of a 6-periodic sequence $-1,-1,0,1,1,0,\cdots$. Applying the discrete $\delta$, see that each $s_d$ corresponds to $\pm \sin$ or $\pm \cos$. We formalize this observation, which summarizes the computation of subsections 4.2 and 4.4 in \cite{doyle2022polynomialsrationalpreperiodicpoints}:
\begin{lemma}\label{lem:6periodic}
    Let $\sigma:\mathbb{Z}\to \mathbb{Z}$ be a 6-periodic function defined by $\sigma(0) = 0$, $\sigma(1) = 1$, $\sigma(2)=1$ and $\sigma(m+3) = -\sigma(m)$ for all $m$. Then $s_d(m)$ agrees with $\sigma(m+\frac{3}{2}(d-1))$ on values $m=-\frac{d+1}{2},-\frac{d-1}{2},\cdots,\frac{d+1}{2}$ for all $d\geq 1$.
\end{lemma}
\begin{proof}
    This proof rephrases the definition of $s_d$, which was given in terms of the doubly periodic sequence $\rho(m,d)$ in \cite{doyle2022polynomialsrationalpreperiodicpoints}, in a more intuitive way. Consider the following table, where each row is 6-periodic and each column is 4-periodic:
    \begin{center}
        \begin{tabular}{c|c|c|c|c|c|c|c|c|c|c|c|c|c|c|c}
             \diagbox{$d$}{$m$} & $\cdots$ & $-3 $& $-\frac{5}{2}$ & $-2$ & $-\frac{3}{2}$ & $-1$ & $-\frac{1}{2}$ & 0 & $\frac{1}{2}$ & 1 & $\frac{3}{2}$ & 2 & $\frac{5}{2}$ & 3 & $\cdots$ \\ \hline \hline
               
            1 & & 0 & & $-1$ & & $-1$ & & 0 & & 1 & & 1 & & 0 &\\ \hline

            2 & & & $1$ & & 0 & & $-1$ & & $-1$ & & 0 & & $1$ & \\ \hline
            
            3 & & 0 & & $1$ & & $1$ & & 0 & & $-1$ & & $-1$ & & 0 &\\ \hline

            4 & & & $-1$ & & 0 & & $1$ & & $1$ & & 0 & & $-1$ & \\ \hline

            5 & & 0 & & $1$ & & $1$ & & 0 & & $-1$ & & $-1$ & & 0 &\\ 

        \end{tabular}
    \end{center}
    Note that applying the central difference operator $\delta$ to the $(d+1)$-th row gives the $d$-th row. Unpacking the definition given in \cite{doyle2022polynomialsrationalpreperiodicpoints}, it is straightforward to check that $\rho(m,d) = \sigma(m+d-2)$. Since $s_d$ is defined by the interpolation $s_d(m-\frac{d+1}{2}) = \rho(m,d)$ on integers $0\leq m\leq d+1$, and the $d$-th row is given by values of the function $\sigma(m+\frac{3}{2}(d-1))$, it follows that $s_d$ interpolates the centremost $d+2$ values of the $d$-th row above.
    
\end{proof}
The family $s_d$ hence serves as an example of a family of polynomials with small range (i.e. within $-1,0,1$) over an interval of size $d+2$; in fact, $s_d$ has a linearly (in $d$) bounded range over an interval of size $d+6$. The transformation from $s_d$ to $r_d$ ensures that $r_d$ is integer-valued on $\mathbb{Z}$ and dynamically compresses the interval $[d+6]$.

\section{Estimates for $s_d$}

In this section we prove a series of estimates of $s_d$ which allow us to study the H\'enon maps of Section \ref{section:henon}.  As an application of these estimates to one-dimensional dynamics, we will also see in Subsection \ref{subsection:ratl_preper_rd} that the dynamically compressing maps $r_d$ have no rational preperiodic points outside the set of integers they compress; that is, the bounds given for the explicit examples of Section 4 of \cite{doyle2022polynomialsrationalpreperiodicpoints} are sharp.

\subsection{Local uniform convergence of $s_d$}

In this subsection we establish the local uniform convergence of subsequences of $s_d$ to the trigonometric functions they approximate.

\begin{proposition}
    $(-1)^d s_{2d+1}(z)$ locally uniformly converges to the sine function $\frac{2}{\sqrt{3}}\sin\left(\frac{\pi z}{3}\right)$ as $d\to \infty$ on the complex plane; that is,     
    $$ \sum_{j=0}^{\infty}(-1)^j c_{2j+1}(z) = \frac{2}{\sqrt{3}}\sin\left(\frac{\pi z}{3}\right)$$
    locally uniformly on $z\in \mathbb{C}$.
    \label{lem:odd_sd_unifcvg}
\end{proposition}
\begin{proof}
    Observe that
		\begin{align*}
			(-1)^j c_{2j+1}(z) 
			&= (-1)^j z \frac{(z+1)(z+2)\cdots(z+j)(z-1)(z-2)\cdots (z-j)}{(2j+1)!} \\
			&= z \frac{(1-z)(2-z)\cdots (j-z)(1+z)(2+z)\cdots (j+z)}{(\frac{3}{2})(\frac{5}{2})\cdots(\frac{2j+1}{2})}\frac{1}{4^j j!} \\
			&= z \frac{(1-z)^{(j)}(1+z)^{(j)}}{(3/2)^{(j)}}\frac{1}{4^j j!},
		\end{align*}
		where $q^{(j)} := q(q+1)\cdots (q+j-1)$ denotes the rising factorial. Hence the infinite sum of these functions are
		$$ \sum_{j=0}^{\infty}(-1)^j c_{2j+1}(z)
		 = z\sum_{j=0}^{\infty} \frac{(1-z)^{(j)}(1+z)^{(j)}}{(3/2)^{(j)}}\frac{1}{4^j j!}
		= z{}_2F_1\left(1-z,1+z;\frac{3}{2};\frac{1}{4}\right), $$
		where ${}_2F_1$ is the \textit{Gaussian hypergeometric function} defined by
		$$ {}_2F_1(a,b;c;z) = \sum_{n=0}^{\infty}\frac{a^{(n)}b^{(n)}}{c^{(n)}}\frac{z^n}{n!} $$
		for $c\not\in \mathbb{Z}_{\leq 0}$ and $|z| < 1$. Note that the ratios between consecutive terms converge to
        $$ \frac{(a+n)(b+n)z}{(c+n)n} \to z,$$
        which implies that the series converges uniformly on compact subsets of
        $$ \left\{ (a,b,c,z)\in \mathbb{C}^4: c\not\in \mathbb{Z}_{\leq 0},\ |z|<1\right\}. $$
        We use formula 15.1.16 of \cite{Abramowitz}:
		$$ {}_2F_1\left(a,2-a;\frac{3}{2};\sin^2 w\right) = \frac{\sin( (2a-2)w)}{(a-1)\sin(2w)}. $$
		Putting $a=1-z$ and $w=\pi/6$, we get the desired result of
		$$ \sum_{j=0}^{\infty}(-1)^j c_{2j+1}(z)
		 = z{}_2F_1\left(1-z,1+z;\frac{3}{2};\frac{1}{4}\right) = \frac{2}{\sqrt{3}}\sin\left(\frac{\pi z}{3}\right), $$ 
         where the convergence is locally uniform in $z\in \mathbb{C}$.
\end{proof}

Since all $s_d$ and the sine function are holomorphic, it follows from standard results in complex analysis that the derivative of $(-1)^d s_{2d+1}(z)$ also locally converges to the derivative of the sine wave. A similar argument shows that $(-1)^d s_{2d}(z)$ and its derivative locally uniformly converge to $\frac{2}{\sqrt{3}}\cos(\frac{\pi z}{3})$ and its derivative on the complex plane.

\subsection{Tail behaviour of $s_d$}
We now work with $s_d$ as functions of a real (rather than complex) variable and show that $s_d$ is strictly increasing on the real interval $[\frac{d+3}{2},\infty)$; this is achieved by first showing that $s_d$ is strictly increasing on $[\frac{d+3}{2},\frac{d+5}{2}]$ and then applying a double induction argument.

For our analysis, we  need some explicit auxiliary bounds on the derivative of $c_d$ on the real line. For convenience, we put $c_{d}(x)$ in the same setting for both even and odd $d$ by use of the generalized binomial coefficient, which we denote 
$$ \tilde{c}_d(x):= \frac{1}{d!}x(x-1)(x-2)\cdots (x-d+1), $$
noting that $c_{d}(x) = \tilde{c}_d(x+\frac{d-1}{2})$.

 \begin{lemma}
    $|c_{d}'(x)| \leq 1 $ for $|x|\leq \frac{d-1}{2}$.
\label{lem:first_cd_derivative_bound}
\end{lemma}
\begin{proof}
    We prove that $|\tilde{c}_d'(x)| \leq  1$ on $0\leq x\leq d-1$. By the product rule, we have the following expression:
    $$ \tilde{c}_d'(x) = \frac{1}{d!}\sum_{0\leq j< d}^{}
    \ \prod_{0\leq i < d,i\neq j}^{}(x-i). $$
    By the triangle inequality,
    $$ |\tilde{c}_d'(x)| 
    \leq \frac{1}{d!}\sum_{0\leq j< d}^{}
    \ \prod_{0\leq i < d,i\neq j}^{} |x-i| 
    \leq \frac{1}{d!}\sum_{0\leq j< d}^{} (d-1)! = 1. $$
    
\end{proof}
	 
\begin{lemma}
    $|c_{d}'(x)| \leq  \frac{6+4\log d}{d(d-1)}$ for $|x|\leq \frac{d-3}{2}$.
\label{lem:second_cd_derivative_bound}
\end{lemma}
\begin{proof}
    We prove that $|\tilde{c}_d'(x)| \leq \frac{6+4\log d}{d(d-1)} $ on $1\leq x\leq d-2$. By continuity of $\tilde{c}_d'$, it suffices to consider non-integer $x$. We carry on from the expression above:
    $$ |\tilde{c}_d'(x)| 
    \leq \frac{1}{d!} \sum_{0\leq j< d}^{}
    \ \prod_{0\leq i < d,i\neq j}^{} |x-i|.$$
    Fix an integer $1\leq N<d-2$ and take $N<x< N+1$. Note that for $j< N$,
    $$ \prod_{i\neq j}^{} |x-i| \leq \frac{(N+1)!(d-N-1)!}{x-j}
    \leq \frac{(N+1)!(d-N-1)!}{N-j}.$$
    Similarly for $j> N+1$, 
    $$ \prod_{i\neq j}^{} |x-i| \leq \frac{(N+1)!(d-N-1)!}{j-x}
    \leq \frac{(N+1)!(d-N-1)!}{j-N-1}.$$
    For $j=N$ and $N+1$, we simply get
    $$ \sum_{j=N,N+1}^{}
    \ \prod_{i\neq j}^{} |x-i| \leq ((x-N)+(N+1-x))\cdot \prod_{i\neq N,N+1}|x-i|\leq (N+1)!(d-N-1)!.$$
    Combining these bounds, we obtain:
    \begin{align*}
        |\tilde{c}_d'(x)| &\leq \frac{1}{d!}(N+1)!(d-N-1)!(1+H_{N-1}+H_{d-N}) \\
        &\leq \frac{1+2H_d}{ {}_{d}C_{N+1}} \leq \frac{1+2(\log d+1)}{d(d-1)/2} = \frac{6+4\log d}{d(d-1)},
    \end{align*}
    where $H_n =1+1/2+\cdots+1/n < \log n + 1$.
\end{proof}

\begin{lemma}
    $s_d(x)$ is strictly increasing on $\frac{d+3}{2}\leq x \leq \frac{d+5}{2}$ for all $d\geq 1$.

\label{lem:sd_increasing_intvofsize1}
 \end{lemma}
 \begin{proof}
    For $d\leq 44$, this is verified by numerical calculation, so it suffices to prove that $s_d'(x) > 0 $ on $[\frac{d+3}{2},\frac{d+5}{2}]$ for all $d> 44$. 
    
    We established that
    $$ s_d'(x) - c_{d+2}'(x) + c_{d+4}'(x) - c_{d+6}'(x) + \cdots =  \begin{cases}
        \frac{2\pi}{3\sqrt{3}}\cos\left(\frac{\pi x}{3}\right) & d \text{ odd} \\
        -\frac{2\pi}{3\sqrt{3}}\sin\left(\frac{\pi x}{3}\right) & d \text{ even}
    \end{cases} $$
    where the sum converges uniformly on $\frac{d+3}{2}\leq x \leq \frac{d+5}{2}$. Moreover note that 
    \begin{align*}
        |c_{d+6}'(x)-c_{d+8}'(x)+\cdots|
         &\leq |c_{d+6}'(x)| + \sum_{k=4}^{\infty}|c_{d+2k}'(x)|\\
         &< 1 + \sum_{k=4}^{\infty}\frac{6+4\log(d+2k)}{(d+2k)(d+2k-1)} =: E_d	 
    \end{align*}
    by Lemmas \ref{lem:first_cd_derivative_bound} and \ref{lem:second_cd_derivative_bound}. Since each of the summands are decreasing in $d$ and numerical calculation gives $E_d<2$ for $d= 44$, it follows that $E_d<2$ for all $d> 44$. Hence it suffices to show that
    $$ c_{d+2}'(x) - c_{d+4}'(x) > \frac{2\pi}{3\sqrt{3}} + 2 $$
    on $\frac{d+3}{2}\leq x \leq \frac{d+5}{2}$. Equivalently, it suffices to show that $$ \tilde{c}_{d+2}'(x)-\tilde{c}_{d+4}'(x+1) > \frac{2\pi}{3\sqrt{3}} + 2 $$
    on $d+3\leq x\leq d+4$. Expanding the LHS expression with the product rule,
    
    \begin{align*}
        &\tilde{c}_{d+2}'(x) - \tilde{c}_{d+4}'(x+1) \\
        &= \frac{d}{dx}\frac{1}{(d+2)!}\left( x(x-1)\cdots (x-d-1) \left(1-\frac{(x+1)(x-d-2)}{(d+3)(d+4)}\right) \right) \\
        &= \frac{1}{(d+2)!}\left(1-\frac{(x+1)(x-d-2)}{(d+3)(d+4)}\right)
        \sum_{j=0}^{d+1}\prod_{i\neq j,0\leq i\leq d+1}^{}(x-i) \\
        &-\frac{1}{(d+2)!}x(x-1)\cdots (x-d-1)\left(\frac{2x-d-1}{(d+3)(d+4)}\right).
    \end{align*}
    We can bound the expressions above on $d+3\leq x\leq d+4$ as follows:
    \begin{itemize}
        \item By putting $x=d+4$,
        $$ 1-\frac{(x+1)(x-d-2)}{(d+3)(d+4)} \geq 1 - \frac{2(d+5)}{(d+3)(d+4)}
        \geq\frac{d}{d+3}. $$
        \item By putting $x=d+3$,
        \begin{align*}
            \sum_{j=0}^{d+1}\prod_{i\neq j,0\leq i\leq d+1}^{}(x-i) 
            &\geq (d+3)!\left(\frac{1}{2}+\frac{1}{3}+\cdots + \frac{1}{d+3}\right)
            > (d+3)! \left(\log d -1/2 \right),
        \end{align*}
        where we use the lower bound $H_d > \log d + 1/2$.
        \item By putting $x=d+4$,
        $$ \frac{1}{(d+2)!}x(x-1)\cdots (x-d-1) 
        \leq \frac{(d+4)(d+3)}{2}. $$
        \item By putting $x=d+4$,
        $$ \frac{2x-d-1}{(d+3)(d+4)} \leq \frac{d+7}{(d+3)(d+4)} < \frac{2}{d+3}. $$
        
    \end{itemize}
    Combining all of these bounds, we get the following bound:
    \begin{align*}
        \tilde{c}_{d+2}'(x) - \tilde{c}_{d+4}'(x+1) 
        &> d\left(\log d-1/2\right) - d-4
    \end{align*}
    on $d+3\leq x\leq d+4$. Since $d\left(\log d-1/2\right) - d-4 > \frac{2\pi}{3\sqrt{3}}+2$ for all $d> 44$, this proves the lemma.
 \end{proof}

 \begin{proposition}
    $s_{d}(x)$ is strictly increasing on $x \geq \frac{d+3}{2}$ for all $d\geq 1$.

    \label{prop_sd_strict_increasing}
 \end{proposition}
 \begin{proof}
 We proceed by induction on $d$. Claim is true for $d=1$ since $s_d(x) = x$. Suppose the claim is true for $d-1$. We already established that $s_d(x)$ is increasing on $[\frac{d+3}{2},\frac{d+5}{2}]$ by Lemma \ref{lem:sd_increasing_intvofsize1}. For $\frac{d+5}{2}\leq x<y\leq \frac{d+7}{2}$, observe that
    \begin{align*}
        s_d(y) &= s_d(y-1) + s_{d-1}(y-1/2) \\
        &> s_d(x-1) + s_{d-1}(x-1/2) = s_d(x),
    \end{align*}
    since $\delta s_{d} = s_{d-1}$ holds by definition, and $s_{d-1}$ is strictly increasing on $[\frac{d+2}{2},\infty)$ by induction hypothesis. Thus $s_d$ is strictly increasing on $[\frac{d+5}{2},\frac{d+7}{2}]$. Repeating this process for $[\frac{d+7}{2},\frac{d+9}{2}]$ and so on, we conclude that $s_d$ is strictly increasing on $[\frac{d+3}{2},\infty)$. Hence the claim is true for $d$. By induction, the claim is true for all $d$.
 \end{proof}

We conclude this subsection with a lower bound of $s_d$ on some interval in the integers. This result will be useful in proving the boundedness of $K(r_d)$ and understanding the preperiodic points of $r_d$, and will also be used in Section 4 when analyzing the periodic points of the H\'enon map $h_d$.

\begin{lemma}
    For all $d\geq 3$ and $x\in \mathbb{Z} + \frac{d+1}{2}$, we have $ s_d(x) \geq 3x$ for all $x\geq \frac{d+7}{2}$.

    \label{lem:3x_lowerbound_on_sd}
\end{lemma}

\begin{proof}
    We proceed by induction on $d$. The claim is true for $s_3(x) = \frac{x^3-7x}{6}$. Now take $d\geq 4$ and suppose claim is true for some $d-1$. Take $x=\frac{d+7+2k}{2}$ with integer $k\geq 0$, then by $s_{d-1} = \delta s_d$ and the induction hypothesis we obtain
    \begin{align*}
        s_{d}(x) 
        &= s_{d}\left(\frac{d+5}{2}\right) + \sum_{j=0}^{k}s_{d-1}\left(\frac{d+6+2j}{2}\right) \\
        &
        \geq s_{d}\left(\frac{d+5}{2}\right) + \sum_{j=0}^{k}\frac{3}{2}(d+6+2j) 
        = s_{d}\left(\frac{d+5}{2}\right) + \frac{3}{2}(k+1)(d+6+k).
    \end{align*}
    Since $s_d(\frac{d+5}{2})\geq d+1$ by Lemma 4.2(5) of \cite{doyle2022polynomialsrationalpreperiodicpoints}, we have
    \begin{align*}
        s_d(x) - 3x &\geq d+1 + \frac{3}{2}(k+1)(d+6+k) - \frac{3}{2}(d+7+2k) \\
        & = d-\frac{1}{2}+\frac{3}{2}k(k+d+5) \geq \frac{9}{2} + \frac{3}{2}k(k+9) > 0
    \end{align*}
    for all $k\geq 0$ and $d\geq 4$. Therefore the claim is true for $d$.
\end{proof}

\subsection{Rational preperiodic points of $r_d$} \label{subsection:ratl_preper_rd}

As a consequence of our estimates in the previous subsections, we will show that the effective examples $r_d$ of \cite{doyle2022polynomialsrationalpreperiodicpoints} have no `extra' rational preperiodic points.

\begin{proposition}
    Let $K(f)\subset \mathbb{C}$ denote the filled Julia set of polynomial $f$. We have that $K(r_d)\cap \mathbb{R} \subset (0,d+7)$ for all $d\geq 2$.
    \label{thm:rd_escape_radius}
 \end{proposition}
 \begin{proof}
    It suffices to show that $|r_d(x)| > |x|$ for $x\geq d+7$ and $x\leq 0$. This is true for $d=2$, so wlog assume $d\geq 3$.
    \begin{itemize}
        \item $d$ even: Under conjugation by translation, we must show that
        $$s_d(x) - \frac{d+3}{2} > |x| $$
        for $|x|\geq \frac{d+7}{2}$. By symmetry of $s_d$, it suffices to show this for $x\geq \frac{d+7}{2}$ only. Recall from Lemma \ref{lem:3x_lowerbound_on_sd}  that $s_d(n) \geq 3n$ for all $n\in \mathbb{Z}+1/2$ with $n\geq \frac{d+7}{2}$. Taking any real number $n\leq x< n+1$, we have 
        $$ s_d(x) \geq s_d(n) 
        \geq 3n 
        > x + \frac{d+3}{2}. $$
        Note that we use Proposition \ref{prop_sd_strict_increasing} for $s_d(x)\geq s_d(n)$.
        
        \item $d$ odd: Under conjugation by translation, we must show that $ s_d(x) - x-1 > x $ for $x \geq \frac{d+7}{2}$ and $ s_d(x)-x-1 < x $ for $x\leq -\frac{d+7}{2}$. By symmetry of $s_d$, it suffices to show that
        $$ s_d(x) > 2x + 1. $$
        for $x \geq \frac{d+7}{2}$. Noting that $s_d(n) \geq 3n$ for all integers $n \geq \frac{d+7}{2}$, for any real number $n\leq x < n+1$ we have
        $$ s_d(x) \geq s_d(n) \geq 3n > 2x+1. $$
        Again note that we use Proposition \ref{prop_sd_strict_increasing} for $s_d(x)\geq s_d(n)$.

    \end{itemize}
 \end{proof}

 We now consider the non-Archimedean behavior of $r_d(x)$ to complete our study of the set of rational preperiodic points for $r_d(x)$.  From the closed form for $s_d(x)$ we will show that for any prime $p$, the $p$-adic escape radius of $r_d$ is bounded above by $1$.

 \begin{proposition} Let $d \geq 2$. Let $p$ be a prime, with associated $p$-adic absolute value $|\cdot|_p$.  If $x \in \mathbb{Q}$ with $|x|_p > 1$ then $|r_d(x)|_p > |x|_p$.
 \end{proposition}

\begin{proof} We first consider the case when $d = 2n$ is even.  Recall then that
$$r_d(x) = s_d\left(x-3-\frac{d+1}{2}\right) + 2,$$
where
$$s_{d}(x) = \sum_{k=0}^{n}(-1)^{n-k}c_{2k}(x)$$
and
$$c_{2k}(x) = \frac{1}{(2k)!}\prod_{j=1}^{k} \left(x^2-\left(\frac{2j-1}{2}\right)^2\right).$$
Suppose that $|x|_p > |1/2|_p,$ and write $y =x-3-\frac{d+1}{2}.$  Then $|y|_p = |x|_p > |1/2|_p$ and so $\left| y^2 - \left( \frac{2j-1}{2} \right)^2 \right|_p = |y|_p^2$, so 
$$|c_{2k}(y)|_p = |y|_p^{2k} |(2k)!|_p^{-1}.$$
Since $|x|_p > 1,$ this is strictly increasing in $k$ and so we see that 
$$|r_d(x)|_p =|s_d(y)|_p = |c_{2n}(x)|_p \geq |y|^{2n}_p = |x|_p^{2n}> |x|_p$$
as claimed.  It remains to deal with the case when $p = 2$ and $|x|_2 = 2.$  In that case, $x - 3 - (d+1)/2$ is an integer, and we have
$$|c_{2k}(x)|_2 = |1/4|^{k}_2 |(2k)!|^{-1}_2.$$
Again this increases in $k$ and so 
$$|r_d(x)|_2 = |s_d(x-3-(d+1)/2)|_2 = |c_{2n}(x-3-(d+1)/2)|_2 = |1/4|^{n}_2 |(2n)!|^{-1}_2 > |x|_2,$$
as desired.

The case when $d$ is odd is a similar computation, noting that we use $d \geq 2$ to ensure that $|s_d(x-3-(d+1)/2)|_p > |x|_p$ whenever $|x|_p > 1.$

\end{proof}

\begin{corollary} Let $d \geq 2$ and $r_d(x)$ as defined above.  Then the set of rational preperiodic points of $r_d(x)$ is precisely ${1, 2, \dots, d+6}$.
\end{corollary}

\begin{proof} By the previous proposition, any rational preperiodic point is a $p$-adic integer for all primes $p$, thus a rational integer.  Since the filled Julia set satisfies $K(r_d) \cap \mathbb{R} \subset (0,d+7)$, the only possible rational preperiodic points lie in the set ${1, 2, \dots, d+6},$ and as $r_d(x)$ maps this set to itself (Theorem 4.4, \cite{doyle2022polynomialsrationalpreperiodicpoints}), this is precisely the set of rational preperiodic points for $r_d(x)$.
\end{proof}

\subsection{Optimality among dynamically compressing polynomials}

In Section 5 of \cite{doyle2022polynomialsrationalpreperiodicpoints}, the authors list integer-valued polynomials of degree $2\leq d\leq 20$ that exhibit dynamical compression of large intervals. The example provided for $d=2$ coincides with the polynomial $r_2$ and compresses $r_2([8])\subseteq [7]$. By elementary means, one can show that this is the optimal choice for dynamical compression in degree $2$ in the following sense: if an integer-valued quadratic polynomial $f$ of degree 2 satisfies $f([m])\subset [m]$, then $m\leq 8$. In fact, $m=8$ is attained by only $r_2(x)=\frac{x^2-9x+22}{2}$ and $r_2(x)+1$.

In contrast, $r_3$ is does not compress the maximal possible range in degree $3$ since it compresses $r_3([9])\subseteq [7]$, while $f(x) =\frac{x^3-18x^2+89x-66}{6}$, which is the example listed in \cite{doyle2022polynomialsrationalpreperiodicpoints}, compresses $f([11])\subseteq [11]$. Nevertheless, it is possible by similar elementary means to show that $f$ compresses the largest interval for degree 3 and is the unique such polynomial.  Unfortunately, the elementary approach does not generalize well even to other small degrees.

\section{Hénon maps with many periodic points} \label{section:henon}

In this section, we study the dynamical behaviour of the following family of H\'{e}non maps:
$$ h_d(x,y) = (y,-x+s_d(y)) $$
for odd $d\geq 3$. Recall that $s_d$ is integer-valued for odd $d$, so $h_d$ (and its inverse) preserves the integral lattice $\mathbb{Z}^2$. This family exhibits many interesting properties, namely having a large number of periodic integer points near the origin with long periodic cycles.  

We are interested in asymptotic results so focus on the case of $d$ odd to simplify the estimates and exposition: similar results will likely hold for some suitably chosen shift in the case of $d$ even.

Note that $h_d$ is conjugate to its inverse:
$$ h_d^{-1}(x,y) = (-y+s_d(x),x) = 
(r\circ h_d \circ r^{-1})(x,y), $$
where $r(x,y) = (y,x)$ is a map that reflects the $x$ and $y$ coordinates. 

\subsection{Escape radius}

In this subsection, we investigate the (archimedean and non-archimedean) filled Julia sets of the Hénon map $h_d$. The \emph{filled Julia set} $K(f)$ of a plane polynomial automorphism $f:\mathbb{C}^2\to \mathbb{C}^2$ is defined to be $K(f) = K^+(f) \cap K^-(f)$ where
    $$ K^{\pm}(f) = \left\{z\in \mathbb{C}^2:\sup_{n\geq 1}\|f^{\pm n}(z)\| < \infty \right\}.$$
In particular, the filled Julia set $K(f)$ contains all periodic points of $f$.

We prove that there exists some $R_v > 0$ for each rational place $v$ such that every periodic point of $h_d$ must satisfy $\max\left\{|x|_v,|y|_v\right\} < R_v$ for all $v$. The proof involves cutting up $\mathbb{Q}^2$ into subsets between which the dynamics of the points are well-understood; it is a similar approach to studying the filtration properties of Hénon maps as demonstrated in \cite{Bedford1991}. The escape radii results will imply that all periodic points of $h_d$ are integer points within a certain box and leads to an upper bound on the number of periodic points on $h_d$ as a quadratic function of $d$.

For each place $v$ of $\mathbb{Q}$, we define the following norm on $\mathbb{Q}^2$:
$$ \|(x,y)\|_{v} = 
    \max\left\{|x|_v,|y|_v\right\}. $$
Let $s\in \mathbb{Q}[x]$ be a non-zero odd polynomial of degree at least two, and define 
$$h(x,y) = (y,-x+s(y)).$$
For each place $v$, choose $R_v > 0$ such that $|s(y)|_v > 2|y|_v$ for all $|y|_v \geq R_v$. 

\begin{lemma}
    For each rational place $v$, consider the sets:
    $$ \mathcal{D}_v 
    = \left\{(x,y) \in \mathbb{Q}^2:
    |y|_v\geq |x|_v \geq  R_v \right\},\ 
    \mathcal{T}_v = \left\{(x,y)\in \mathbb{Q}^2: 
    |x|_v<R_v\leq |y|_v\right\}
    $$
    For all $(x,y)\in \mathcal{D}_v\cup \mathcal{T}_v$, we have $h(x,y)\in \mathcal{D}_v$ and $\|h(x,y)\|_{v} > \|(x,y)\|_{v}$.
\end{lemma}
\begin{proof}
    For point $(x,y)\in\mathcal{D}_v\cup \mathcal{T}_v$, note that $|s(y)|_v >2|y|_v$ and $\|(x,y)\|_v = |y|_v$. Applying the triangle inequality, we obtain:
    \begin{align*}
        |-x + s(y)|_v \geq 2|y|_v - |x|_v  > |y|_v \geq R_v.
    \end{align*}
    Hence $h(x,y)\in \mathcal{D}_v$ and $\|h(x,y)\|_v > \|(x,y)\|_v$.
\end{proof}
The lemma above shows that the $\|\cdot\|_v$-norm of $(x,y)\in\mathcal{D}_v\cup\mathcal{T}_v$ strictly increases under iterations of $h$, which implies that all points $(x,y)\in \mathcal{D}_v\cup\mathcal{T}_v$ are not periodic points of $h$. The following diagram describes the dynamics of points under $h$:
\begin{center}
    \begin{tikzpicture}
        \draw(-0.5, 0) edge[->] (4,0);
        \draw(0,-0.5) edge[->] (0,4);
        \node at (4.3, 0) {$|x|_v$};
        \node at (0, 4.3) {$|y|_v$};
        \draw[dashed] (1,-3pt)--(1,4);
        \draw[dashed] (-3pt,1)--(4,1);
        \draw[dashed] (-3pt,-3pt)--(4,4);
        \node at (1,-0.5) {$R_v$};
        \node at (-0.5,1) {$R_v$};
        \fill[blue, nearly transparent] (3.8,3.8) -- (1,1)--(1,3.8);
        \fill[red, nearly transparent] (1,3.8) -- (1,1) -- (0,1) -- (0,3.8);
        \node at (2, 3.2) {\huge $\circlearrowright$};
        \draw[line width=0.7pt, ->] (0.5, 2.7) to[bend left=60] (1.3, 3.2);
        \node[red] at (0.5, 2.2) {$\mathcal{T}_v$};
        \node[blue] at (1.55, 2.2) {$\mathcal{D}_v$};
    \end{tikzpicture}
\end{center}
Since $h^{-1}=r\circ h\circ r$ for $r(x,y)=(y,x)$, it follows that the $\|\cdot\|_v$-norm of $(x,y)\in r(\mathcal{D}_v\cup\mathcal{T}_v)$ strictly increases under iterations of $h^{-1}$. Therefore any point in $\mathcal{D}_v\cup\mathcal{T}_v\cup r(\mathcal{D}_v\cup\mathcal{T}_v) $ cannot be a periodic point of $h$. Since this set is precisely the complement of the open ball of radius $R_v$ under the $\|\cdot\|_v$-norm, we reach the following conclusion:
\begin{proposition}
    If $(x,y)\in \mathbb{Q}^2$ is a periodic point of $h$, then $\|(x,y)\|_v < R_v$ for all places $v$ of $\mathbb{Q}$.
\end{proposition}

We now apply this to the case of $s = s_d$ and $h = h_d$ for odd $d \geq 3$. We wish our bounds to depend explicitly on $d$, which amounts to finding for each place some $R_v>0$ such that $|s_d(y)|_v > 2|y|_v$ for all $|y|_v\geq R_v$.

For $v = \infty$, recall from the proof of Theorem \ref{thm:rd_escape_radius} that $s_d(y) > 2y + 1$ for $x\geq \frac{d+7}{2}$. It follows from $s_d$ being an odd function that $|s_d(y)| > 2|y|$ for all $|x|\geq \frac{d+7}{2}$. Hence we can take $R_{\infty} = \frac{d+7}{2}$.

For $v = p$ prime, observe that the leading coefficient of $s_d$ is precisely $1/d!$ and that all coefficients are integer multiplies of $1/d!$. Take $|y|_p > 1 + 2|d!|_p$, then by applying the ultrametric inequality on $|s_d(y)|_p$, we obtain:
$$ |s_d(y)|_p - 2|y|_p \geq \left|\frac{1}{d!}\right|_p|y|_p^d - \left|\frac{1}{d!}\right|_p|y|_p^{d-1} - 2|y|_p^{d-1} 
= |y|_p^{d-1}\frac{|y|_p - 1-2|d!|_p}{|d!|_p} > 0.$$
Hence we can take $R_p = 1 + 3|d!|_p$. Using these results, we conclude:

\begin{corollary}
For odd $d\geq 3$, every rational periodic point $(x,y)$ of $h_d$ is an integer point satisfying $\max\left\{|x|,|y|\right\} \leq \frac{d+5}{2}$.
\end{corollary}
\begin{proof}
    By $R_{\infty} = \frac{d+7}{2}$, every rational periodic point $(x,y)$ of $h_d$ satisfies $\max\left\{|x|,|y|\right\} < \frac{d+7}{2}$. It remains to show that $(x,y)$ is an integer point. Observe that it suffices to prove $R_p < p$, since $\|(x,y)\|_p < R_p < p$ for all $p$ implies that both $x$ and $y$ are integers. This inequality is true for all $(p,d)$ with the exception of $(p,d)=(2,3)$. For this case, we have $\|(x,y)\|_2 < 1 + 3\cdot|3!|_2 < 2^2$ and thus an exhaustive check of all half-integer points with $\max\left\{|x|,|y|\right\} < \frac{3+7}{2}=5$ is sufficient.
\end{proof}

\subsection{Counting periodic points}
In the previous section, we showed that every rational periodic point of $h_d$ must be an integer point inside the box $\|(x,y)\|_{\infty} \leq \frac{d+5}{2}$. We now turn to show that most integer points inside this box are periodic points of $h_d$.
\begin{theorem}
    Define $R$ to be as follows:
    \begin{equation*}
        R:=\begin{cases}
            \frac{d+1}{2}  & \textnormal{if } d\equiv 1 \mod 6 ,\\
            \frac{d+1}{2} - 2& \textnormal{if } d\equiv 3\mod 6 ,\\
            \frac{d+1}{2} -3 & \textnormal{if } d\equiv 5\mod 6 .
        \end{cases}
    \end{equation*}
    Any integer point $(x,y)\in \mathbb{Z}^2$ with $\|(x,y)\|_{\infty} \leq R$ is periodic under $h_d$.
\end{theorem}
\begin{proof}
    Define the box $B = \left\{ (x,y)\in \mathbb{Z}^2:\|(x,y)\|_{\infty}\leq R \right\}$. It is sufficient to prove for $(x,y)\in B$ that the iterates eventually return to $B$. Recall that $|s_d(y)|\leq 1$ for $|y|\leq \frac{d+1}{2}$, hence for $(x,y)\in B$ we have
    $$ \|h_d(x,y)\|_{\infty} \leq |x| + |s_d(y)|\leq R+1.$$
    Here equality holds iff $x = R$, $s_d(y) = -1$ or $x = -R$, $s_d(y) = 1$. In any other case, we are done since we immediately get $h_d(x,y)\in B$.
    
    For the two cases above, observe that $h_d(-x,-y) = -h_d(x,y)$ and so it suffices to consider the case of $x=R$, $s_d(y) = -1$. The computation is tedious, but since values of $s_d$ over the integers $-\frac{d+5}{2},\cdots,\frac{d+5}{2}$ are completely understood given $d\mod 6$, we can divide into cases and manually iterate $h_d$ starting from the point $h_d(x,y) = (y,-R-1)$.  As an example, suppose that $d \equiv 1 \mod 6$ and $y \equiv 4 \mod 6$.  Recall by Lemma \ref{lem:6periodic} that $s_d$ is $6$-periodic from $-R$ to $R$, with $s_d(R-1) = 0, s_d(R) = 1$ and that $s_d(R+1) = 2.$ We have
    $$(y, -R-1) \mapsto (-R-1,-y-2)$$
    since $s_d$ is odd and $s_d(R+1) = 2.$  If $|-y-2| > R$ and $|y| \leq R$ we must have $y \in \{ R-1, R \}$ which contradicts $s_d(y) = -1.$ So $|-y-2| \leq R$ and the value of $s_d(-y-2) = 0.$ So 
        $$(-R-1,-y-2) \mapsto (-y-2, R+1) \mapsto (R+1, y+4).$$
    If $|y+4| \leq R$ then $(R+1, y+4) \mapsto (y+4, -R)$ and we have returned to the box $B$.  If $|y+4| > R$ then we must have $y = R-3$ as $y \equiv 4 \mod 6.$  But we compute directly the orbit of $(R-3, -R-1)$:
        \begin{align*}
        & (R-3,-R-1) \mapsto (-R-1,-R+1) \mapsto (-R+1,R+1) \mapsto (R+1,R+1) \\
        & \mapsto (R+1,-R+1) \mapsto (-R+1,-R-1) \mapsto (-R-1,R-3) \mapsto (R-3,R),
    \end{align*}
    and $(R-3,R) \in B.$ Proceeding in this way for all cases of $d$ and $y$, one sees that the iterates return to $B$ after finitely many times.
\end{proof}
From the two results above, we deduce the following.
\begin{corollary}
    The number of rational periodic points of $h_d$ satisfies:
    $$ (d-4)^2 \leq \text{Per}(h_d,\mathbb{Q}) \leq (d+6)^2. $$
\end{corollary}

\begin{remark}
    Computations suggest that the actual number of integer periodic points of $h_d$ is:
    \begin{equation*}
        \text{Per}(h_d,\mathbb{Q})
        = 
        \begin{cases}
            d^2 -\frac{8d}{3}+\frac{56}{3} & \textnormal{if }d\equiv 1 \mod 6 ,\\
            d^2 + 8 & \textnormal{if }d\equiv 3 \mod 6 ,\\
            d^2 -\frac{8d}{3}+\frac{40}{3} & \textnormal{if }d\equiv 5 \mod 6.
        \end{cases}
    \end{equation*}
\end{remark}

\section{Long cycles}

In this section, we investigate long periodic cycles present within the integer periodic points of $h_d$. This is the most surprising aspect of this family of H\'enon maps, as it turns out that for $d\equiv 1\mod 6$ the length of the longest cycle is $\frac{8d+10}{3},$ substantially better than interpolation in the degree $d$ H\'enon family provides. With small modifications, similar results can be obtained in any odd degree, though we do not provide proof.

\begin{theorem}\label{thm:thmBinsection}
    For $d\equiv 1\mod 6$, there exists a periodic cycle of length $\frac{8d+10}{3}$.
\end{theorem}
\begin{proof}
    For $d\equiv 1\mod 6$ recall that $R = \frac{d+1}{2}$. We first tabulate the tail values of $s_d$ as follows:
    \begin{center}
        \begin{tabular}{c||c|c|c|c|c}
            $y$ & $R-2$ & $R-1$ & $R$ & $R+1$ & $R+2$ \\ \hline
            $s_d(y)$ & $-1$ & 0 & 1 & 2 & $d+2$
        \end{tabular}
    \end{center}
    We also recall that $s_d$ is an odd function and that it is 6-periodic on $-R,\cdots,R$ with values $-1,-1,0,1,1,0$. This determines the values of $s_d$ on integers $-R-2$ to $R+2$. We claim that the orbit of $(R+1,-R+1)$ under iterations of $h_d$ take the following form:
    \begin{center}
            \adjustbox{scale=0.8,center}{
            \begin{tikzcd}
                {(-R-1,-R+1)} & {(-R+1,R+1)} & {(R+1,R+1)} & {(R+1,-R+1)} \\
                {(-R-1,-R+7)} &&& {(R+1,-R+7)} \\
                \vdots &&& \vdots \\
                {(-R-1,R-7)} &&& {(R+1,R-7)} \\
                {(-R-1,R-1)} & {(-R-1,-R-1)} & {(R-1,-R-1)} & {(R+1,R-1)}
                \arrow["8", maps to, from=1-4, to=2-4]
                \arrow["8", maps to, from=2-4, to=3-4]
                \arrow["8", maps to, from=3-4, to=4-4]
                \arrow["8", maps to, from=4-4, to=5-4]
                \arrow["1", maps to, from=5-4, to=5-3]
                \arrow["1", maps to, from=5-3, to=5-2]
                \arrow["1", maps to, from=5-2, to=5-1]
                \arrow["8", maps to, from=5-1, to=4-1]
                \arrow["8", maps to, from=4-1, to=3-1]
                \arrow["8", maps to, from=3-1, to=2-1]
                \arrow["8", maps to, from=2-1, to=1-1]
                \arrow["1", maps to, from=1-1, to=1-2]
                \arrow["1", maps to, from=1-2, to=1-3]
                \arrow["1", maps to, from=1-3, to=1-4]
            \end{tikzcd}
    }
    \end{center}
    Here the numbers on the arrows denote the number of iterations of $h_d$ involved. The horizontal arrows are clear. For the right-hand side vertical arrows, it suffices to prove that $h_d^{\circ 8}(R+1,y) = (R+1,y+6)$ for integer $y$ satisfying $-R+1\leq y\leq R-7$ and $y\equiv R-1\mod 6$. This is a simple check:
    \begin{align*}
        & (R+1,y) \mapsto (y,-R-1) \mapsto (-R-1,-y-2) \mapsto (-y-2,-R) \mapsto (R,y+3) \\
        & \mapsto (y+3,-R) \mapsto (-R,-y-4) \mapsto (-y-4,R+1) \mapsto (R+1,y+6)
    \end{align*}
    The left-hand side vertical arrows follow from  $h_d(-x,-y) = -h_d(x,y)$. Therefore the period of $(R+1,-R+1)$ is
    $$ 6 + 2\cdot 8\cdot \frac{2R-2}{6} = \frac{8d+10}{3}.$$
    \end{proof}

    As noted in the introduction, the family of H\'enon maps is defined by $d+2$ parameters, so we see that these cycles are much longer than one can hope for from a naive interpolation argument (see Section 11 of \cite{henon_openqs} for details of this argument).

  \subsection{Effects of shifting}
    In the case when $d\equiv 3,5\mod 6$, computations suggest that the longest periodic cycle of $h_d$ has length $\leq 20$ for all $d$. This resonates with the fact, which will be discussed later in the last section, that the longest cycle of the limiting H\'enon map $h_{\infty}$ has length 20.  It is an interesting feature of the dynamics on the boundary of the escape region in the case $d \equiv 1 \mod 6$ that $h_d$ is a good but not \emph{too} good approximation of the map $h_{\infty}$. This suggests that for $d \equiv 3,5 \mod 6$ similar long cycles might be obtained by minor (integer-valued) perturbations of $h_d.$  Indeed, here we list some experimental results that arise from shifting our H\'enon maps.  Define $h_{d,c}(x,y)$ to be
$$ h_{d,c}(x,y) =(-y,x+s_d(y+c)) $$
where $c\in \mathbb{Z}$. We find that such H\'enon maps may have even more integer periodic points and even longer cycles for a suitably chosen shift $c$. The following tables contain polynomial interpolations of computational results obtained for $15\leq d\leq 299$; we do not provide any proof.

\definecolor{LightGreen}{rgb}{0.80,0.95,0.80}
\begin{itemize}
    \item $d\equiv 1\mod 6$:
    \vskip 0.5em
    \begin{center}
        \begin{tabular}{|c|c|c|}
            \hline 
            $c$ & Number of integer periodic points & Longest cycle \\
            \hline \hline
            $ 0$    & $d^2 -\frac{8d}{3}+\frac{56}{3}$  & $\frac{ 8d+10}{3}$    \\ \hline
            $\pm 1$ & \cellcolor{LightGreen!60}$d^2 + 2d +  4$ & $\frac{10d- 7}{3}$    \\ \hline
            $\pm 2$ & $d^2 - 6d + 18$                   & \cellcolor{LightGreen!80}$\frac{16d-61}{3}$    \\ \hline
        \end{tabular}
    \end{center}
    \vskip 1em
    \item $d\equiv 3\mod 6$:
    \vskip 0.5em
    \begin{center}
        \begin{tabular}{|c|c|c|}
            \hline 
            $c$ & Number of integer periodic points & Longest cycle \\
            \hline \hline
            $ 0$ & $d^2 + 8$  & $20$    \\ \hline
            $-1$ & $d^2 + 4d$   & \cellcolor{LightGreen!80}$8d-39$ \\ \hline 
            $ 1$ & \cellcolor{LightGreen!80}$d^2 + 4d + 1$  & {\cellcolor{LightGreen!80}$8d-39$}   \\ \hline
            $\pm 2$ & $d^2 - 4d + 7$  & $60$    \\ \hline
        \end{tabular}
    \end{center}
    \vskip 1em
    \item $d\equiv 5\mod 6$:
    \vskip 0.5em
    \begin{center}
        \begin{tabular}{|c|c|c|}
            \hline 
            $c$ & Number of integer periodic points & Longest cycle \\
            \hline \hline
            $ 0$    & $d^2 -\frac{8d}{3}+\frac{40}{3}$  & $20$    \\ \hline
            $\pm 1$ &\cellcolor{LightGreen!80} $d^2 - 2d +  29$                   & $\frac{14d- 31}{3}$    \\ \hline
            $\pm 2$ &$d^2 -\frac{22d}{3}+\frac{161}{3}$ & \cellcolor{LightGreen!80}$\frac{28d-185}{3}$    \\ \hline
        \end{tabular}
    \end{center}
\end{itemize}

\subsection{Transcendental Hénon maps} \label{subsection:transc_henon}

Recall from Proposition \ref{lem:odd_sd_unifcvg} that $(-1)^d s_{2d+1}(x)$ locally uniformly converges to $\frac{2}{\sqrt{3}}\sin(\frac{\pi x}{3})$ on the real line as $d$ tends to infinity. We conclude this article by studying the dynamical behaviour of the limiting Hénon map:
$$ h_{\infty}(x,y) = \left(y, -x + \frac{2}{\sqrt{3}}\sin\left(\frac{\pi y}{3}\right)\right). $$
As each $h_{d}$ is integer-valued for each odd $d$, the rapid convergence of $s_d$ on a neighborhood of $0$ ensures that the integer periodic points of $h_d$ will have dynamics described by $h_{\infty}$ for points of sufficiently small norm compared to $d$.

\begin{proposition}
		Any integer point $(x,y)\in \mathbb{Z}^2$ is periodic under iteration of $h_{\infty}$ with periods 1,4,5,6,12 or 20. The period is determined by the values of $x,y$ modulo 6 as follows:
		\begin{center}
			\[
			\begin{array}{@{}l*{7}{c}@{}}
				\toprule
				y\mod 6 & \multicolumn{6}{c@{}}{x\mod 6}\\
				\cmidrule(l){2-7}
				& 0 & 1 & 2 & 3 & 4 & 5\\
				\midrule
				5 & 12 & 20 & 20 & 20 & 20 & 12\\
				4 & 20 & 20 & 12 & 12 & 20 & 20\\
				3 & 4 & 20 & 12 & 4 & 12 & 20\\
				2 & 20 & 20 & 20 & 12 & 12 & 20\\
				1 & 12 & 12 & 20 & 20 & 20 & 20\\
				0 & 4 & 12 & 20 & 4 & 20 & 12\\
				\bottomrule
			\end{array}
			\]
		\end{center}
		There are the following exceptions:
		$$ \text{period of }(x,y) = 
		\begin{cases}
			1 & (x,y) = (0,0), \\
			5 & (x,y) = (2,0),(2,1),(1,2),(0,2),(-1,1),\\
			&(-2,0),(-2,-1),(-1,-2),(0,-2),(1,-1), \\
			6 & (x,y) = (1,0),(1,1),(0,1),(-1,0),(-1,-1),(0,-1). \\
		\end{cases} $$

  \label{prop_transc_period}
\end{proposition}
\begin{proof}
		One can check this manually by iterating the point $(6x+a,6y+b)$ for integers $0\leq a,b< 6$. Here we show only one instance of this. Take $a=1$, $b=0$. The forward orbit of $(6x+1,6y)$ over $h_{\infty}$ is as follows:
		\begin{align*}
			&(6x+1,6y)\mapsto (6y,-6x-1)\mapsto (-6x-1,6y-1) \mapsto (-6y-1,6x) \mapsto (6x,6y+1)\\
			& \mapsto (6y+1,-6x+1) \mapsto (-6x+1,-6y) \mapsto (-6y,6x-1)\mapsto (6x-1,6y-1) \\
			&\mapsto (6y-1,-6x) \mapsto (-6x,-6y+1)\mapsto (-6y+1,6x+1) \mapsto (6x+1,6y).
		\end{align*}
		Note that all of these points are distinct \textit{except} for the case of $x=y=0$, when the sixth iterate is equal to the initial point, hence the point $(1,0)$ (and its iterates) having a period of six.
\end{proof}

We now analyse the local behaviour of $h_d$ near the integer points. Observe that the Jacobian matrix of $h_{\infty}$ is 6-periodic in both $x$ and $y$:
$$ Dh_{\infty}(x,y) = \begin{pmatrix}
    0 & 1 \\ 
    -1 & \frac{2\pi}{3\sqrt{3}}\cos(\frac{\pi x}{3})
\end{pmatrix}. $$
Note also from the proof of Proposition \ref{prop_transc_period} that the forward orbit modulo 6 only depend on the values of $x$ and $y$ modulo 6, excluding some $(x,y)$ that are close to the origin. The two facts combined makes it possible to explicitly calculate the multiplier and its eigenvalues for all integer points.

Here we instead take a visual approach: We give a small perturbation on each integer point and plot its forward orbit. Each points in the forward orbit is coloured as follows. Forward orbits of points near:
\begin{itemize}
    \item period 1 points are coloured in bright red;
    \item period 4 points are coloured in dark red/black;
    \item period 6 points are coloured in bright blue;
    \item period 12 points are coloured in navy/cyan;
    \item period 5 points are coloured in green/dark purple;
    \item period 20 points are coloured in lime/bright purple.
\end{itemize}
After applying $2 \times 10^5$ iterations of $h_{\infty}$ to randomly perturbed points near integer points, we obtain the following picture:
\begin{center}
    \includegraphics[width=1\textwidth]{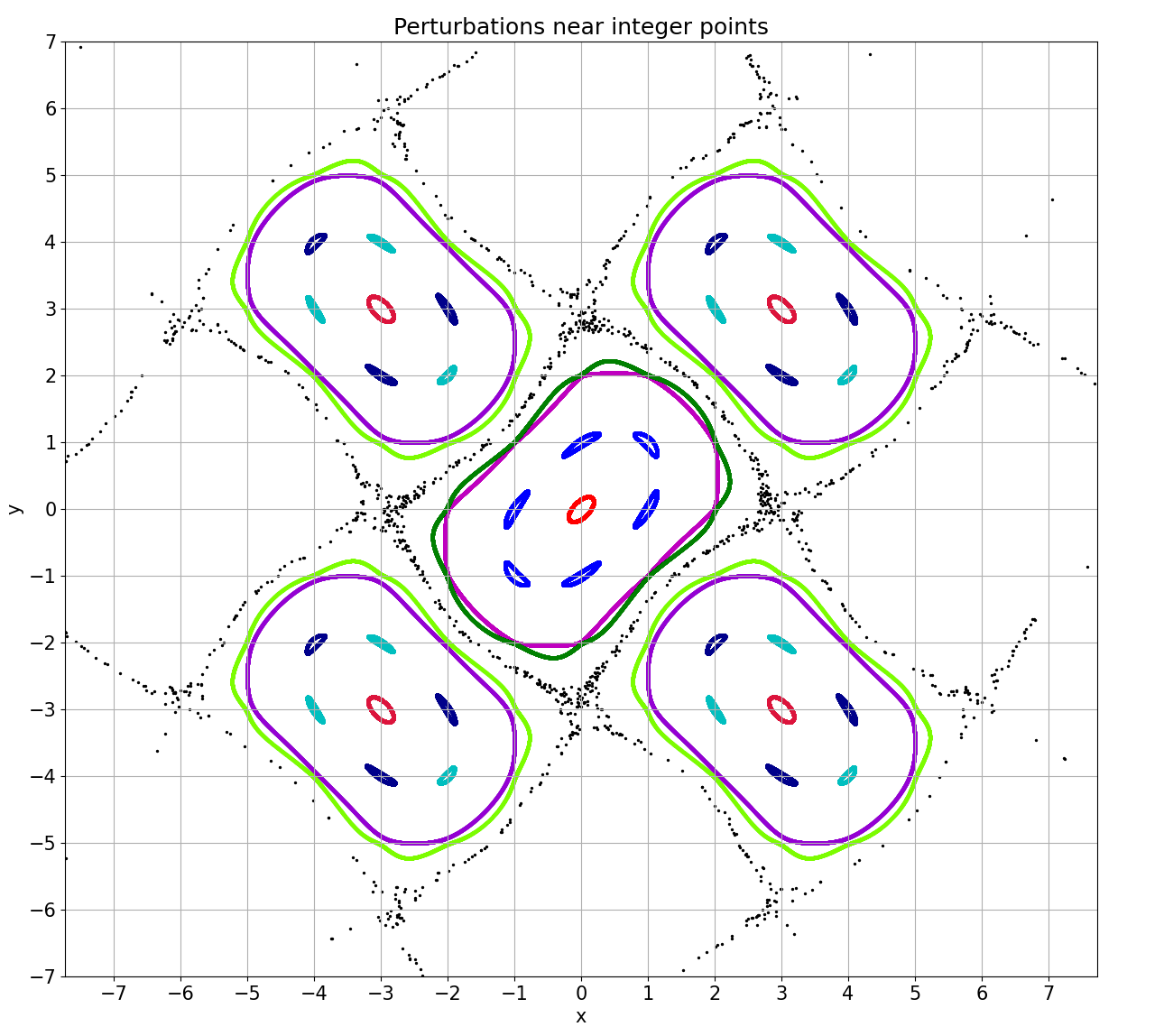}
\end{center}
In the same way as $h_d$, all integer points have multipliers with determinant 1. Hence the eigenvalues are either (i) real numbers whose product is 1 (of \textit{saddle type}), or (ii) complex conjugates with modulus 1.

Looking at the picture above, we can categorize the integer points as follows:
\begin{itemize}
    \item Period 6, 12 points and period 4 points $(x,y)\in 3\mathbb{Z}\times 3\mathbb{Z}$ with $x-y\equiv 0\mod 6$ are not of saddle type (the orbits near these period 4 points are coloured in dark red). The small closed loops formed around those points suggest that the perturbed points are indeed being rotated.
    \item Period 5, 20 points are of saddle type. The orbits of perturbations of these points converge to two distinct limit sets, each coloured in green and purple.
    \item Period 4 points $(x,y)\in 3\mathbb{Z}\times 3\mathbb{Z}$ with $x-y\equiv 3 \mod 6$ are of saddle type, and the orbits near these points are coloured in black. The following picture depicts the forward orbit of a perturbation near $(3,0)$ after $5\times 10^5$ iterations of $h_{\infty}$:
    \begin{center}
        \includegraphics[width=0.75\textwidth]{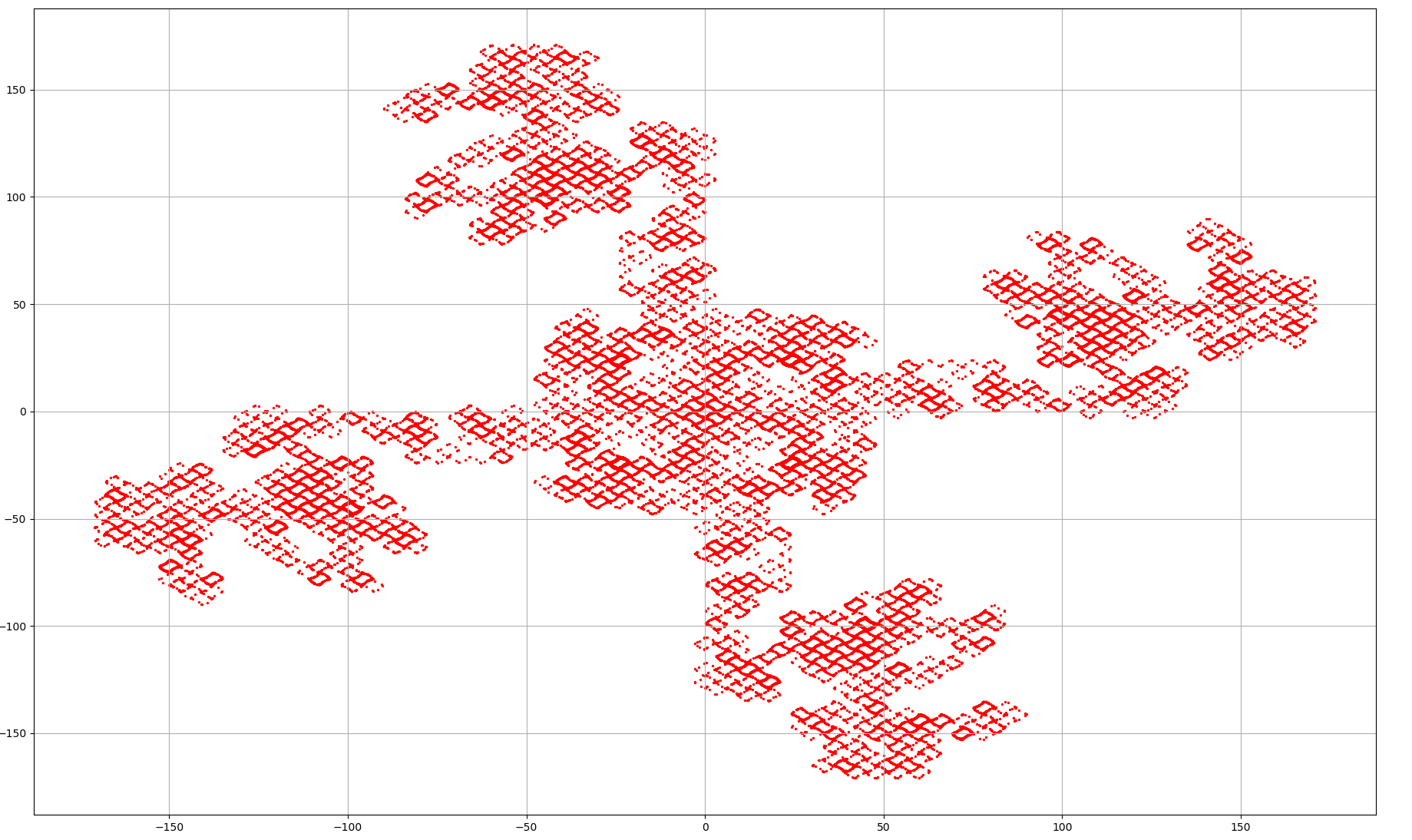}
    \end{center}
\end{itemize}
 
One might naturally ask whether a better choice of limiting map might produce longer cycles; however, it may be that any similar construction of a family of H\'enon maps arising from a convergent sequence of integer-valued polynomials (analogous to $s_d$) can only admit long cycles on the boundary of the escape region, which would imply a linear bound in cycle length.  It would also be interesting to explore the family of H\'enon maps of the form $(x,y) \mapsto (y, -x+f_d(y))$ where $f_d$ is the (not explicit) family of polynomials with $d + \lfloor \log_2(d) \rfloor$ rational preperiodic points whose existence is known by Doyle and Hyde.

\newpage 

\printbibliography

\end{document}